\documentclass[a4paper,11pt]{amsart}
\usepackage{amsmath,amsthm,amsfonts,amssymb,mathrsfs}
\usepackage{enumerate}
\usepackage[colorlinks=true, linkcolor=blue, citecolor=magenta, menucolor=black]{hyperref}
\usepackage{geometry}
\usepackage[demo]{graphicx}
\usepackage[up]{caption}
\usepackage{subcaption}
\usepackage{pgf,tikz}
\usepackage[toc,page]{appendix}
\usetikzlibrary{arrows}
\usetikzlibrary{patterns}
\usepackage{color}
\usepackage{tikz-cd}

\numberwithin{equation}{section}
\theoremstyle{plain}
\newtheorem{thm}{Theorem}[section]

\newtheorem{lem}[thm]{Lemma}

\theoremstyle{definition}

\newtheorem*{defn*}{Definition}

\newtheorem{rem}[thm]{Remark}

\newcommand{\N}{\mathbb{N}}

\newcommand{\R}{\mathbb{R}}

\newcommand{\cF}{\mathcal{F}}

\newcommand{\cO}{\mathcal{O}}

\newcommand{\cS}{\mathcal{S}}

\newcommand{\Aut}{\operatorname{Aut}}

\allowdisplaybreaks

\begin{document}

\title{Maximal amenable MASAs of radial type \\ in the free group factors }

\author{R\'emi Boutonnet}
\address{Institut de Math\'ematiques de Bordeaux \\ CNRS \\ Universit\'e Bordeaux I \\ 33405 Talence \\ FRANCE}
\email{remi.boutonnet@math.u-bordeaux.fr}
\thanks{RB Supported by ANR grant AODynG 19-CE40-0008} 

\author{Sorin Popa}
\address{Math Dept, UCLA, Los Angeles CA 90095-1555, USA}
\email{popa@math.ucla.edu} 
\thanks{SP Supported by NSF Grant DMS-1955812  and the Takesaki Endowed Chair at UCLA}

\maketitle

\begin{abstract} 
We prove that if  $\{(M_j, \tau_j)\}_{j\in J}$ are tracial von Neumann algebras,  $s_j \in M_j$ are selfadjoint semicircular elements 
and $t=(t_j)_j$ is a square summable $J$-tuple of real numbers with 
at least two non-zero entries, then the von Neumann algebra $A(t)$ generated by the ``weighted radial element'' $\sum_j t_j  s_j\in M:=*_{j\in J} M_j$ 
is maximal amenable in $M$, with 
$A(t)$, $A(t')$  unitary conjugate in $M$ iff $t, t'$ are proportional. 
Letting $M_j$ be diffuse amenable,  $\forall j$, this provides  a large family of maximal amenable MASAs in  the free group factor $L\mathbb F_n$. 
\end{abstract}

\section{Introduction}

While the isomorphism problem for the free group factors $L\mathbb F_n$, $2\leq n \leq \infty$, 
remains unsolved, a series of absorption, in-decomposability and randomness phenomena in $L\mathbb F_n$, and more generally 
in amalgamated free product II$_1$ factors $M=M_1*_B M_2$, unravelled over the years. 

The simplest type of absorption result shows that  whenever a subalgebra $Q$ of $M_1$ ``stays away'' from $B$, $Q\nprec_{M_1} B$, 
any element in $M=M_1*_B M_2$ that intertwines  $Q$ into $M_1$ 
must be contained in $M_1$ (\cite{Po81} in the case $B=\Bbb C$,   \cite{IPP05} in general). While elementary, this method 
did for instance allow to deduce in-decomposability, such as primeness and absence of Cartan subalgebras, in 
free group factors $L\mathbb F_J$, with uncountable set $J$ of generators (see \cite{Po81}). 

Another type of absorption, triggered by the discovery in \cite{Po82} that a diffuse amenable 
subalgebra $Q\subset M$  that's freely complemented in $M$ is maximal amenable in $M$, has been much studied over the years. 
A general such amenable absorption result in \cite{BH16} shows that if $Q\subset M_1$ is maximal amenable in $M_1$ 
and $Q\nprec_{M_1} B$, then $Q$ is maximal amenable in $M=M_1*_B M_2$ (see also \cite{HJNS19} for similar such results obtained via free entropy methods). 

A random matrix approach to in-decomposability was initiated in \cite{V94}, based on the discovery that 
 ``thinness over MASA'' assumptions on $L\mathbb F_n$ come in contradiction with 
randomness, quantified by the free entropy dimension. 
This allowed proving that $L\mathbb F_n$ have no Cartan subalgebras \cite{V94}, primeness \cite{G96},  
and more generally absence of ``thinness'' with respect to diffuse AFD-subalgebras in $L\mathbb F_n$ \cite{GP96}. 

Deformation-rigidity and boundary methods have been used to establish even  stronger absorption results in $L\mathbb F_n$,  
more generally in amalgamated free product factors (see e.g., \cite{P01,O03,IPP05,Po06,OP07,PV11,I13,DKP22,Dr22}), 
showing for instance that the normalizer of any diffuse amenable subalgebra 
of $L\mathbb F_n$ generate an amenable von Neumann algebra \cite{OP07}. 

Motivated by some of these results, as well as by their work in $L^2$-cohomology of groups, 
Peterson-Thom made in \cite{PT07} the far reaching conjecture that if $Q_i\subset L\mathbb F_n$, $2\leq n \leq \infty$, 
is a family of amenable von Neumann subalgebras with $\cap_i Q_i$ diffuse, 
then $\vee_i Q_i$ is amenable. Equivalently, any maximal amenable $Q\subset L\mathbb F_n$  ``absorbs''  any other 
amenable $Q_0\subset L\mathbb F_n$ that has diffuse intersection with $Q$,   
a phenomenon superseding all above absorption statements. 

This conjecture, and even stronger ones formulated later (in \cite{H15,P18}), is now settled in the affirmative due 
to the combination of two results: on the one hand, Hayes used 1-bounded entropy \cite{J07} to prove that if one could establish the strong convergence of tuples of tensor products of gaussian unitary ensemble (GUE) random matrices, see \cite{H20},
then the PT-conjecture holds true; on the other hand, Belinschi-Capitaine  \cite{BC22} and Bordenave-Collins \cite{BoCo23} 
have recently proved this strong convergence result.  

Note that by \cite{Dy93} if $Q$ is a diffuse amenable tracial von Neumann algebra and $N$ is a tracial von Neumann algebra 
that's either a free group factor $L\mathbb F_n$, $2\leq n \leq \infty$, or a diffuse amenable tracial von Neumann algebra, then $M=Q*N$ 
is isomorphic to the free group factor $L\mathbb F_{m}$, where $m=2$ when $N$ is amenable, and $m=n+1$ when $N=L\mathbb F_n$. 
So by \cite{Po82}, in all these cases $Q$ follows maximal amenable in $M=Q*N$, with the 
proof actually showing that any amenable von Neumann subalgebra of $M$ that has diffuse intersection with $Q$ 
must be contained in $Q$. Thus, whenever an amenable diffuse 
von Neumann subalgebra $Q\subset M$ 
is freely complemented, it does satisfy the PT-absorption conjecture by \cite{Po82}.  

While there are by now other examples of maximal amenable von Neumann 
subalgebras in $L\mathbb F_n$, 
one could not establish whether they are freely complemented or not. This fact is amply discussed in   
\cite[Section 5]{P18}. But a resolution of the PT-conjecture 
makes it now particularly  compelling to answer the question: are there actually any maximal amenable von Neumann subalgebras of $L\mathbb F_n$ 
that are not freely complemented ? It would of course be a very  striking structural property of $L\mathbb F_n$ 
if  all of its maximal amenable von Neumann subalgebras are freely complemented.   
But this is a possibility that should not be ruled out. 

The so-called {\it radial MASA} $L_n$ in $L\mathbb F_n$, $2\leq n <\infty$,  
defined as the abelian von Neumann algebra generated by the ``radial element'' 
$\sum_{i=1}^n (u_i+u_i^*)$, where $u_1, ..., u_n$ are the canonical unitaries corresponding to the free generators of $L\mathbb F_n$, 
introduced in (\cite{Py81}) and shown in \cite{CFRW09} to be maximal amenable in $L\mathbb F_n$,  
is a good candidate for an example of a non-freely complemented maximal amenable subalgebra in the case $n<\infty$ 
(see \cite[Question 1.1]{CFRW09} and \cite[Question 5.5]{P18}; see also Remark 3.5.$1^\circ$ in \cite{Po82}). 

Our work here is motivated by an effort to produce examples of non-freely complemented maximal amenable von Neumann 
subalgebras in $L\mathbb F_n$. While we could not solve this problem, in this short note 
we produce a large family of distinct maximal amenable MASAs in $M=L\mathbb F_n$, and more generally in free products of factors $M=*_j M_j$. 
Our examples are ``radial-type'' MASAs, but they actually do not recover the ``classic'' radial MASA 
$L_n\subset L\mathbb F_n$. However, various considerations make them good candidates 
for not being freely complemented. 

To state our result we use Voiculescu's notion of {\it semicircular} elements in tracial 
von Neumann algebras and his {\it free Gaussian formalism}, allowing $\ell^2$-summations of freely independent such elements (see \cite{V88}).

\begin{thm}\label{MAMS}
Let $\{(M_j, \tau_j)\}_{j\in J}$ be a family of tracial von Neumann algebras, with a self-adjoint semi-circular element $s_j$ in each $M_j$, $j\in J$. 
Denote by $\ell^2_*$ the set of square summable families of real numbers having at least two non-zero entries. For each $t=(t_j)_j\in \ell^2_{*}$ 
denote by $A(t)$ the abelian von Neumann  generated in $M=*_{j\in J} M_j$ by $s(t):= \sum_j t_j s_j\in M$. Then $A(t)$ is maximal amenable in $M$, $\forall t\in \ell^2_{*}$, 
with $A(t)\prec_M A(t')$ if and only if $t, t' \in \ell^2_{*}$ are proportional. 
\end{thm}

The subordination relation $Q\prec_M P$ for two von Neumann subalgebras $Q, P$ in a tracial von Neumann algebra $M$ is in the sense 
of (Section 2 in \cite{Po03}) and it reads $Q$ {\it can be intertwined into $P$ inside $M$}. 
Of the several equivalent characterizations provided in (2.1 of \cite{Po03}), we will take as definition the one requiring 
the existence of  a non-zero $\xi\in L^2M$ (an {\it intertwiner}) such that the $Q-P$ Hilbert bimodule $\mathcal H=\overline{\text{\rm sp} Q \xi P}$ is finite 
as a right $P$-module, i.e., ${P^{op}}' \cap \mathcal B(\mathcal H)$ is a finite von Neumann algebra. Note that  by multiplying such 
$\xi$ from the right with an appropriate central projection of $P$, this is equivalent 
to the existence of $0\neq \xi' \in L^2M$ such that $\mathcal H'=\overline{\text{\rm sp} Q \xi' P}$ is finitely generated as right $P$-module 
(see also Lemma 2.2 in \cite{Po03}). 

Note that in case $J$ is (at most) 
countable in the above theorem, with $M_j$ either separable amenable or $\simeq L\mathbb F_{n(j)}$, $\forall j\in J$, the resulting $M=*_j M_j$ 
follows of the form $L\mathbb F_{n(J)}$, for some $2\leq n(J) \leq \infty$, by \cite{Dy93}, 
thus making $A(t)\subset M$, $t\in \ell^2_*,$ a large class of examples of maximal amenable MASAs in free group factors. 
Taking $M_j$ ``much larger'' than the abelian algebra $A_j=\{s_j\}''$, generated by the semicircular element $s_j\in M_j$, seems to indicate that $A(t)$ cannot be freely complemented in $M$. For instance, taking $M_j = A_j \rtimes \Gamma_j$, for some trace preserving action of an amenable group $\Gamma_j \curvearrowright A_j$. This heuristic seems particularly pertinent when taking $M_j=A$, $\forall j$, with $A$ abelian non-separable (e.g., an ultrapower of $L^\infty[0,1]$), 
where the statement could perhaps be used towards an existence result in $L\mathbb F_n$, for $n$ finite 
(see the remarks at the end of this paper for more comments along these lines). 

For the proof of the theorem we will use  the absorption results in \cite{Po81, IPP05,BH16} and the trick for ``glueing'' intertwiners from (page 398 in \cite{Po03}). 

First of all, note that we may assume $J$ is at most countable. Indeed, by  (Remark 6.3 and Corollary 4.3  and its proof in \cite{Po81}; 
see also  Corolary 4.3 and its proof in \cite{Po90}), any intertwiner between some $A(t), A(t')$, $t, t'\in \ell^2_*$, lies in the von Neumann algebra generated by $M_j$ with $j$ in the supports of the $J$-tuples $t=(t_j)_j, t'=(t'_j)_j$, which are countable (see  
Theorem 1.1 in \cite{IPP05} for a more general result about controlling intertwiners in amalgamated free product factors, which will become indispensable shortly). 
Also, if $A(t)$ would  be contained in some amenable $B\subset M$ and $b\in B$ is an arbitrary element, then there exists  a countable $J_0\subset J$ that contains the support of $t$, with $b\in M_{J_0}:=*_{j\in J_0} M_j$. Thus, if $A(t)$ is maximal amenable in $M_{J_0}$, then $b\in A(t)$. This shows that if we can show that $A(t)$ is maximal amenable in any $M_{J_0}$ with $J_0\subset J$ countable, then the result follows for all  $J$, thus reducing the theorem to the case $J$ at most countable. 

The rest of the argument makes crucial use of the von Neumann algebra $N \subset M$ generated by the semicircular elements, $N:=\{s_j \mid j \in J \}''$, 
which we describe via the free Gaussian functor \cite{V88}. 
Thus, letting $H=H(J)$ denote the $|J|$-dimensional real Hilbert space, we alternatively view $N$ as the 
von Neumann algebra generated by $s(\xi)$, $\xi \in H$, on $\cF(H)$, where $\cF(H)$ is the full Fock space of $H$ and $s(\xi) \in N$ 
is the semi-circular operator associated to the unit vector $\xi$ in $H$.  

Note that $N\simeq L\Bbb F_{|J|}$ and that $N$ contains all the abelian von Neumann algebras $A(t), t\in \ell^2_{*}$, with $\ell^2_{*}$ 
identifying naturally with the subset of vectors having at least two non-zero coordinates in the $J$-dimensional real Hilbert space 
$H=H(J)$. We view $A(t), t\in \ell^2_{*}$ as part of the larger family of 
abelian von Neumann subalgebras $A(\xi) = \{s(\xi)\}'' \subset N$, $\xi \in H$. 

Given a non-zero vector $\xi \in H$, if we denote by $H_0 \subset H$ the orthogonal complement of $\xi$, and by $N_0$ the von Neumann algebra associated with $H_0$ by the free Gaussian functor, we see that the inclusion $N$ splits as the free product $A(\xi) \subset A(\xi) *N_0$. So $A(\xi)$ is freely complemented in $N$ and is thus maximal amenable inside $N$ by \cite{Po82}.

\begin{lem}\label{downstairs}
Given unit vectors $\xi_1, \xi_2 \in H$, we have $A(\xi_1) \prec_N A(\xi_2)$ iff $A(\xi_1)=A(\xi_2)$ iff $\xi_1 = \pm \xi_2$. 
\end{lem}
\begin{proof} It is clearly sufficient  to show that if  $\xi_1,\xi_2 \in H$ are unit vectors with $\xi_1\neq \pm \xi_2$, then $A(\xi_1) \nprec_N A(\xi_2)$, as all other implications are trivial. 
Note that in proving this, we may assume $|J|\geq 3$. Indeed, because if $A(\xi_1) \nprec_{\tilde{N}} A(\xi_2)$ 
in the larger factor $L\Bbb F_{|J|+1} = \tilde{N}\supset N$, then $A(\xi_1) \nprec_N A(\xi_2)$ as well. 

Consider the set $\cS$ of pairs of unit vectors $\xi_1,\xi_2 \in H$ such that $A(\xi_1) \prec_M A(\xi_2)$. Assume by contradiction that this set contains a pair of non-proportional vectors. We will deduce that $\cS$ must also contain a pair of orthogonal vectors. But this conclusion is absurd because if $\xi_1 \perp \xi_2$ then $A(\xi_1), A(\xi_2)$ are in free position in $N$, so $A(\xi_1)\not\prec_N A(\xi_2)$ by \cite{Po81}.

By functoriality, $\cS$ is invariant under orthogonal transformations: if $(\xi,\eta) \in \cS$ and if $\theta \in \cO(H)$ is an orthogonal transformation, then $(\theta(\xi),\theta(\eta))\in \cS$.

{\bf Claim.} $\cS$ also satisfies the following property: given unit vectors $\xi, \eta, \eta'\in H$ with $\langle \xi, \eta \rangle=\langle \xi , \eta'\rangle$, if $(\xi,\eta) \in \cS$, then $(\eta,\eta') \in \cS)$.

Indeed, if $(\xi,\eta) \in \cS$, then by (Theorem A.1 in \cite{P01}), since $A(\xi)$ and $A(\eta)$ are MASAs in $N$, there exist a non-zero partial isometry $v\in N$ such that $v^*v\in A(\xi)$, $vv^*\in A(\eta)$ and $vA(\xi)v^*=A(\eta)vv^*$. Moreover, by our assumption, there exists $\theta \in \cO(H)$ such that $\theta(\xi) = \xi$ while, $\theta(\eta) = \eta'$. The automorphism $\alpha \in \Aut(N)$ associated with $\theta$ leaves all elements in $A(\xi)$ fixed and satisfies $\theta(A(\eta))=A(\eta')$. Thus
\[\theta(v)A(\xi) \theta(v)^*=\theta(vA(\xi)v^*)=\theta(A(\eta)vv^*)=A(\eta')\theta(vv^*), \]
implying that $w=\theta(v)v^*\in N$ is a partial isometry with right support $w^*w=vv^*\in A(\eta)$, left support $ww^*=\theta(vv^*) \in A(\eta')$ and satisfying $wA(\eta)w^*=A(\eta')ww^*$. So $(\eta,\eta') \in \cS$.

Starting with $(\xi,\eta) \in \cS$, $\xi \neq \pm\eta$, we may find a small rotation $\theta$ fixing $\xi$, such that the angle between $\eta$ and $\eta' := \theta(\eta)$ is of the form $2^{-k}$ for some large enough $k$. So by the claim, we have found a pair of vectors $(\eta,\eta') \in \cS$ forming an angle of the form $2^{-k}$, for some $k \in \N$.

Applying again the claim, we see that if $(\xi,\eta) \in \cS$ is a pair of vectors forming an angle $\alpha$ then there exists $\eta'$ such that $(\eta,\eta') \in \cS$ and which forms an angle $2\alpha$ with $\eta$. So by iterating this procedure, we may reach a pair $(\xi,\eta) \in \cS$ with $\xi \perp \eta$, giving our contradiction.
\end{proof}

With the above notations, we have $A(t) \subset N \subset M$ with the properties stated in the theorem being satisfied in $N$: 
the abelian von Neumann algebras $A(t), t\in \ell^2_*$, are maximal amenable in $N$ and the space of intertwiners in $N$ between any two of them  
is equal to $0$, unless the corresponding $t\in \ell^2_*$ are proportional. 

Thus, in order to prove the theorem, all we need is to ``lift'' these properties 
from $N$ to $M$. To prove that $A(t)$'s are non-intertwinable in $M$ as well, we'll need the absorption result  from \cite{IPP05}.

\begin{lem}\label{going up}
Let $B \subset P_1, B\subset P_2$ be inclusions of tracial von Neumann algebras, with ${\tau_{P_1}}_{|B}={\tau_{P_2}}_{|B}$, 
and let $P = P_1 \ast_B P_2$  denote their amalgamated free product. Let  
$A_1,A_2 \subset P_1$ be von Neumann subalgebras and assume  $A_1 \nprec_{P_1} B$. Then $A_1 \nprec_{P} A_2$ iff $A_1 \nprec_{P_1} A_2$.
\end{lem}
\begin{proof} By (Theorem 1.1 in [IPP05]), any intertwiner from $A_1$ to $P_1$ inside $P$ must lie in $P_1$. That is, if $\xi\in L^2P$ is 
so that the Hilbert $A_1-P_1$ bimodule $\mathcal H:=\overline{\text{\rm sp}A_1\xi P_1}\subset L^2P$ is ``finite over $P_1^{op}$'' 
(i.e., $(P_1^{op})'\cap \mathcal B(\mathcal H)$ is a finite von Neumann algebra), then $\xi\in L^2P_1$. Since $A_2\subset P_1$, 
if $A_1\prec_{P} A_2$ then the corresponding non-zero $A_1-A_2$  intertwiner $\xi\in L^2P$ must lie in $P_1$, showing that $A_1\prec_{P_1} A_2$. 
The converse is of course trivial.  
\end{proof}

\begin{proof}[Proof of Theorem \ref{MAMS}] We may assume $J$ is at most countable, with $|J|\geq 2$, i.e., either $J=\{1, 2, ..., n\}$ for some 
finite $n\geq 2$, or $J=\{1, 2, ... \}$. We already know that $A(\xi), A(\eta)$ are non-intertwinable in $N$ for any $\xi,\eta \in H$ that are not proportional (this includes the case $\eta = s_j$, for some $j \in J$). 

Note that for any inclusion of tracial von Neumann 
algebras $B\subset P$ and $Q$ another tracial von Neumann algebra, we have an obvious identification 
\[(B*Q \subset P*Q) \simeq (B*Q \subset P *_B (B*Q)).\] 

We'll use this for the inclusions $N_i\subset N_{i+1}$, $i\geq 0$, 
where $N_0=N=A(s_1)*A(s_2)* ...$, $N_1=M_1*A(s_2)*A(s_3)*...$, and more generally 
$N_{i}=M_1* ... *M_i*A(s_{i+1}) * ... * A(s_n)$, $\forall i\geq 1$, to write each such inclusion 
as $(N_i\subset N_{i+1})=(N_i \subset N_i*_{A(s_{i+1})} M_{i+1})$. 

We prove by induction on $i$ that $A(t)$ is maximal amenable inside $N_i$ and $A(t) \nprec_{N_i} A(s_i)$ and $A(t) \nprec_{N_i} A(t')$, for all $i \geq 0$, and all $t,t' \in \ell^2_*$ non-proportional.

We already proved the case $i = 0$ 

Let now $i \geq 0$ and assume that the result hold for $i$.
So by Lemma \ref{going up}, this implies $A(t) \nprec_{N_{i+1}} A(t')$, $A(t) \nprec_{N_{i+1}} A(s_j)$, $\forall j$. 
By using this fact for $j=i+1$ we also obtain from \cite{BH16} that $A(t)$ is maximal amenable in $N_{i+1}$.  
 
If $|J|<\infty$, the proof is complete. Otherwise, we need to derive the conclusion in $M$. Since $N_n \nearrow M$, if $A(t), A(t')$ would be interwtinable in $M$, via some non-zero $\xi\in L^2M$, then its expectation onto $N_n$ for some large $n$ is non-zero 
and intertwines $A(t), A(t')$ in $N_n$, a contradiction. This implies $A(t)\not\prec_M A(t')$, whenever $t\in \ell^2_*$, $t'\in H$ are not 
proportional (so including the case $s(t')=s_j, j\in J$). 
About the maximal amenability of $A(t)$ in case where $J$ is infinite, we may in fact use a gathering trick. Since $|J|\geq 2$ and $t$ has at least two non-zero entries, we can split $J$ as a disjoint union of two non-empty sets $J=J_1 \sqcup J_2$ such 
that the support of $t$ intersects both $J_1, J_2$. Thus, we can decompose $s(t)$ as $s(t)=(\sum_{j\in J_1} t_js_j) + (\sum_{j\in J_2} t_js_j)=c_1 \tilde{s}_1 + c_2\tilde{s}_2$, 
with $c_1, c_2\in \mathbb R$ non-zero and $\tilde{s}_i\in \tilde{M}_i:=*_{j\in J_i}M_j$, $i=1,2$, self-adjoint semicircular elements. Thus, 
we can apply the case $|J|=2$ above to deduce that  the von Neumann algebra generated by $s(t)$ in $\tilde{M}_1*\tilde{M}_2$, i.e.,  
$A(t)\subset M$, is maximal amenable. 
\end{proof}

\begin{rem} 
 $1^\circ$ While we could not use Theorem \ref{MAMS} to prove that the free group factors 
$L\mathbb F_n$ contain maximal amenable MASAs that are not freely complemented, we believe the following particular cases 
of \ref{MAMS} deserve further investigation: $(a)$ the case $M_j = A(s_j) \otimes R$, $\forall j\in J$, where $A(s_j)=\{s_j\}''\subset M_j$ and $R$ is the hyperfinite 
II$_1$ factor; $(b)$ the case $M_j=A(s_j)\rtimes \Gamma_j$, where $\Gamma_j$ is an amenable group 
and $\Gamma_j \curvearrowright A_j$ is a trace preserving action, $\forall j\in J$;  $(c)$ 
the case $M_j$ abelian non-separable, e.g., an ultrapower of $L^\infty[0,1]$, $\forall j$. 

When studying these cases, it may be useful to take into account that if  $s_j\in M_j$ are fixed, 
the map $t \mapsto s(t)$ is very smooth, a fact that entails pointwise continuity of $t \mapsto A_t$. It may be possible to use 
this continuity in combination
with the fact that at all points $t\in \ell^2_*$ the algebra $A(t)$ is maximal amenable, while at the  ``singularity points $A(s_j)$'' this fails, 
once $M_j$ is ``much larger'' than the semicircular element $s_j\in M_j$.  

The case $M=A^{*n}$ in $(c)$ above, with $A$ abelian purely non-separable, deserves additional attention because of another interesting problem. 
Namely, since the separable algebras $A(t)$ are shown in \ref{MAMS} 
to be maximal amenable (so in particular MASAs) in $M$ despite $A$ being non-separable, this suggests that
the only way to obtain a purely non-separable MASA in such $M$ is to “re-pack” pieces
of $A_k, 1 \leq k \leq n$, where $A_k$ denotes the $k$'th copy of $A$ in the free product $A^{*n}$, and that consequently 
one could recover $n$ from the isomorphism class of $A^{*n}$, $2\leq n \leq \infty$, 
thus showing that these free product factors are non-isomorphic. We leave this as an open problem. 

$2^\circ$ As we mentioned before, a natural candidate for a non-freely complemented  abelian subalgebra in $L\mathbb F_n$ 
is the radial MASA $L_n$ (shown to be a MASA in \cite{Py81} and  maximal amenable in \cite{CFRW09}). 
Another natural class of candidates are the MASAs $A_g=\{u_g\}''\subset L\mathbb F_n$, where  
$g\in \mathbb F_n$  is so that $g^{\mathbb Z}$ is maximal abelian 
in $\mathbb F_n$, but not freely complemented in $\mathbb F_n$. 
A typical such example is the commutator element $g=aba^{-1}b^{-1} \in \mathbb F_2$. 
It was pointed out in (\cite{Po82}, Remark 3.5.1) that  any such abelian subalgebra $A_g\subset L(\mathbb F_n)$  is maximal amenable, so in particular a MASA.  
But it remains open whether such MASAs are 
freely complemented in $L(\mathbb F_n)$ or not. 

\end{rem}

\end{document}